\documentclass[12pt]{article}

\usepackage[margin=1in]{geometry}  
\usepackage{graphicx}              
\usepackage{amsmath}               
\usepackage{amsfonts}              
\usepackage{amsthm}                
\usepackage{amssymb}

\newtheorem{theorem}{Theorem}[section]

\newtheorem{proposition}[theorem]{Proposition}

\theoremstyle{definition}
\newtheorem{definition}[theorem]{Definition}

\theoremstyle{remark}
\newtheorem{remark}[theorem]{Remark}
\numberwithin{equation}{section}

\begin{document}
\nocite{*}

\title{Perfect 2-colorings of the generalized Petersen graph GP(n,3)}

\author{Hamed Karami\\ 
School of Mathematics \\
Iran University of Science and Technology \\
Narmak, Tehran 16846,
Iran\\
hkarami@alum.sharif.edu
}
\date{}
\maketitle
\begin{abstract}
  In this paper we enumerate the parameter matrices of all perfect 2-colorings of the generalized Petersen graphs $GP(n,3)$, where $n\geq 7$. We also give some basic results for $GP(n,k)$.
  \\
  \vspace*{0.1cm}
  \\
  \vspace*{0.1cm}
  \end{abstract}
  \section{Introduction}
The theory of error-correcting codes has always been a popular subject in group theory, combinatorial configuration, covering problems and even diophantine number theory. So, mathematicians always show a lot of interest in this historical research field. The problem of finding all perfect codes was begun by M. Golay in 1949. Perfect code is originally a topic in the theory of error-correcting codes. All perfect codes are known to be completely regular, which were introduces by Delsarte in 1973. A set of vertices say C of a simple graph is called completely regular code with covering radius 
$\rho$,
if the distance partition of the vertex set with respect C is equitable. Therefore, the problem of existence of equitable partitions in graph is of great importance in graph theory. There is another term for this concept in the literature as "perfect m-coloring".\\
As explained above, enumerating parameter matrices in graphs is a key problem to find perfect codes in graphs. For example, by the results of this paper, we can easily conclude that the graph
$GP(9,3)$
has just two nontrivial completely regular codes with the size of 9, which neither of them is perfect. There has always been a notably interest in enumerating parameter matrices of some popular families of graphs "johnson graphs", "hypercube graphs" and recently "generalized petersen graphs" (see \cite{8,9,10,1,2,5,4,6,3}).\\
In this article, all parameter matrices of 
$GP(n,3)$
are enumerated.\\
\textbf{Acknowledgments.} The author would like to thank Yazdan Golzadeh for giving comment on the second part of the theorem \ref{2}. He is also thankful to Mehdi Alaeiyan for his help to enumerating parameter matrices of perfect 2-colorings of 
$GP(n,2)$ graphs, where we obtained some results that were useful for this paper too. 
\section{Definition and Concepts}
In this section, some basic definitions and concepts are given.
\begin{definition}
The \textit{generalized petersen graph}
$GP(n,k)$, also denoted
$P(n,k)$, for 
$n\geq 3$
and
$1\leq k  < \dfrac{n}{2}$,
is a connected cubic graph that has vertices, respectively, edges given by
\begin{center}
$V ( GP(n,k)  )=\lbrace a_i,b_i: 0 \leq i \leq n-1\rbrace ,\qquad\qquad\quad$\\
$E ( GP(n,k) )= \lbrace a_ia_{i+1}, a_ib_i, b_ib_{i+k}: 0\leq i \leq n-1\rbrace$,
\end{center}
\end{definition}
These graphs were introduced by Coxeter (1950) and named by Watkins (1969). $GP(n,k)$
is isomorphic to 
$GP(n,n-k)$. It is why we consider
$k<\dfrac{n}{2}$, with no restriction of generality.
\begin{definition}
For a graph
$G$
and an integer
$m$, a mapping 
$T:V(G)\rightarrow \lbrace 1,\cdots ,m\rbrace$
is called a perfect $m$-coloring with matrix
$A=(a_{ij})_{i,j\in \lbrace 1,\cdots ,m\rbrace }$,
if it is surjective, and for all 
$i,j$,
for every vertex of color
$i$,
the number of its neighbors of color
$j$ is equal to 
$a_{ij}$. The matrix
$A$
is called the \textit{parameter matrix}
of a perfect coloring. In the case $m = 2$,
we call the first color \textit{white}, and the second color \textit{black}. 
\end{definition}
\begin{remark}
In this paper, we consider all perfect 2-colorings, up to renaming the colors; i.e, we identify the perfect 2-coloring with the matrix
$$\begin{bmatrix}
a_{22}& a_{21}\\
a_{12} & a_{11}
\end{bmatrix},
$$
obtained by switching the colors with the original coloring.
\end{remark}
\section{The Existence of Perfect 2-Colorings of $GP(n,3)$}
In this section, we first present some results concerning necessary conditions for the existence of perfect 2-colorings of $GP(n,k)$ graphs with a given parameter matrix $A=(a_{ij})_{i,j=1,2}$, and then we enumerate the parameters of all perfect 2-colorings of $GP(n,2)$.
\\
The simplest necessary condition for the existence of a perfect 2-colorings of $GP(n,k)$ with the matrix 
$\begin{bmatrix}
a_{11}& a_{12}\\
a_{21} & a_{22}
\end{bmatrix}
$ is
$$a_{11}+a_{12}=a_{21}+a_{22}=3.$$
Also, it is clear that neither $a_{12}$ nor $a_{21}$ cannot be equal to zero, otherwise white and black vertices of $GP(n,k)$ would not be adjacent, which is impossible, since the graph is connected.\\
By the given conditions, we can see that a parameter matrix of a perfect 2-coloring of $GP(n,k)$ must be one of the following matrices:
\begin{center}
$
A_1=
\begin{bmatrix}
2 & 1\\
1 & 2
\end{bmatrix},
A_2=
\begin{bmatrix}
2 & 1\\
2 & 1
\end{bmatrix},
A_3=
\begin{bmatrix}
1 & 2\\
2 & 1
\end{bmatrix},$
\\
$A_4=
\begin{bmatrix}
0 & 3\\
1 & 2
\end{bmatrix},
A_5=
\begin{bmatrix}
0 & 3\\
2 & 1
\end{bmatrix},
A_6=
\begin{bmatrix}
0 & 3\\
3 & 0
\end{bmatrix}.$
\end{center}
The next proposition gives a formula for calculating the number of white vertices in a perfect 2-coloring (see \cite{1}).
\begin{proposition}
\label{1}
If $W$ is the set of white vertices in a perfect 2-coloring of a graph $G$ with matrix $A=(a_{ij})_{i,j=1,2}$, then
$$\vert W \vert =\vert V(G) \vert \dfrac{a_{21}}{a_{12}+a_{21}}.$$
\end{proposition}
Now, we are ready to enumerate the parameter matrices of all perfect 2-colorings of $GP(n,3)$. In \cite{8} , the parameter matrices of all perfect 2-colorings of $GP(n,k)$ with the matrices of $A_1$ and $A_6$ are enumerated. So, we just present theorems in order to enumerate parameter matrices corresponding to perfect 2-colorings of $GP(n,3)$ with the matrices $A_2$, $A_3$, $A_4$, and $A_5$.
\subsection*{Perfect 2-colorings of $GP(n,3)$ with the matrix $A_2$:}
In this part, we show that the graphs $GP(n,3)$ have no perfect 2-colorings with the matrix $A_2$.
\begin{theorem}
The graphs $GP(n,3)$ have no perfect 2-colorings with the matrix $A_2$.
\end{theorem}
\begin{proof}
At first, we claim that for each perfect 2-coloring, say $T$, of $GP(n,3)$ with the matrix $A_2$, there are no consecutive vertices $a_i$ and $a_{i+1}$, such that $T(a_i)=T(a_{i+1})=2$. To prove it, suppose contrary to our claim, without loss of generality, there is a perfect 2-coloring, say $T$, of $GP(n,3)$ with the matrix $A_2$, such that $T(a_1)=T(a_2)=2$. It imeadiately gives $T(b_1)=T(b_2)=T(a_0)=T(a_3)=1$ and then $T(b_0)=T(b_3)=T(a_4)=T(b_4)=T(b_5)=1$. Now, from $T(a_3)=T(a_4)=T(b_4)=1$, we have $T(a_5)=2$. Next, from $T(a_4)=T(b_5)=1$ and $T(a_5)=2$, we get $T(a_6)=2$. It gives $T(b_6)=1$ which is a contradiction with $T(a_3)=T(b_3)=T(b_0)=1$.\\
Now, to prove the theorem, suppose the assertion is false. Therefore, there is a perfect 2-coloring, say $T$, of $GP(n,3)$ with the matrix $A_2$. By symmetry, with no loss of generality, we can assume $T(a_0)=T(b_0)=1$ and $T(a_1)=2$. By the above claim, we have $T(a_2)=1$. Now, from $T(a_0)=T(a_2)=1$ and $T(a_1)=2$, it follows that $T(b_1)=2$. This immediately gives $T(b_2)=T(a_3)=T(b_4)=1$ and, in consequence $T(a_4)=1$ and $T(b_3)=T(a_5)=2$. Again, by using the above claim, we get $T(a_6)=1$ and then $T(b_6)=1$, which is a contradiction with $T(a_3)=T(b_0)=1$ and $T(b_3)=2$. 
\end{proof}
\subsection*{Perfect 2-colorings of $GP(n,3)$ with the matrix $A_3$:}
We will show that the graphs $GP(2m,3)$ have a perfect coloring with the matrix $A_3$ and the graphs $GP(2m+1,3)$ have no perfect 2-colorings with the matrix $A_3$.
\begin{theorem}
All of the graphs $GP(n,3)$, where $n$ is even, have a perfect 2-coloring with the matrix $A_3$. Also, there are no perfect 2-colorings of $GP(n,3)$, where $n$ is odd, with the matrix $A_3$.
\end{theorem}
\begin{proof}
To prove the first part, consider the mapping $T:V(GP(2m,3))\rightarrow \lbrace 1,2 \rbrace$ by
\begin{center}
$T(a_{2i}) = T(b_{2i})=1,\quad \;\;\,$\\
$T(a_{2i+1})= T(b_{2i+1})=2.$
\end{center}
for $i \geq 0$. It can be easily seen that the given mapping is a perfect 2-coloring of $GP(2m,3)$ with the matrix $A_3$.\\
To prove the second part, contrary to our claim, suppose there is a perfect 2-coloring, say $T$, of $GP(n,3)$, where $n$ is odd, with the matrix $A_3$. By Lemma 3.4 in \cite{8}, with no loss of generality, we can assume $T(a_0)=T(b_0)=1$. By knowing that the given mapping in the first part is not a perfect 2-coloring with the matrix $A_3$, where $n$ is odd, we should have two cases below.\\
\textbf{Case 1:} For some positive integer $i$, $T(a_i)=T(b_i)=T(b_{i+1})=1$ and $T(a_{i+1})=2$. It immediately gives $T(a_{i+2})=2$ and $T(a_{i+3})=T(b_{i+2})=1$. From $T(a_{i+2})=T(b_i)=1$, we deduce that $T(b_{i+3})=2$ and then $T(a_{i+4})=1$. Next, from $T(b_{i+1})= T(a_{i+4})=1$, we have $T(b_{i+4})=T(a_{i+5})=2$. Now, from $T(b_i)=T(a_{i+3})=1$, and $T(b_{i+3})=2$, it follows that $T(b_{i+6})=2$, and then from $T(a_{i+5})=2$, we get $T(a_{i+6})=1$ and $T(b_{i+5})=2$. Using this argument, for $j \geq 0$, we have
\begin{align*}
T(a_{10j+i})=T(b_{10j+i})=T(b_{10j+i+1})=T(b_{10j+i+2})=T(a_{10j+i+3})=T(a_{10j+i+4})=\qquad\qquad\,\quad &\\ \qquad \qquad \qquad T(a_{10j+i+6})=T(a_{10j+i+7})=T(b_{10j+i+7})=T(b_{10j+i+8})=T(b_{10j+i+9})=1.
\end{align*}
and
\begin{align*}
T(a_{10j+i+1})=T(a_{10j+i+2})=T(b_{10j+i+3})=T(b_{10j+i+4})=T(a_{10j+i+5})=T(b_{10j+i+5})=\qquad\,\,\quad &\\ T(b_{10j+i+6})=T(b_{10j+i+7})=T(a_{10j+i+8})=T(a_{10j+i+9})=2.
\end{align*}
It gives $n=10m$ which contradicts $n$ is odd.\\
\textbf{Case 2:} For some positive integer $i$, $T(a_i)=T(b_i)=T(a_{i+2})=1$ and $T(a_{i+1})=T(b_{i+1})=T(b_{i+2})=2$. It immediately gives $T(a_{i+3})=1$ and then $T(a_{i+4})=T(b_{i+3})=2$. From $T(a_{i+4})=T(b_{i+1})=2$,  we have $T(b_{i+4})=1$ and then we deduce that $T(a_{i+5})=2$ and $T(b_{i+5})=T(a_{i+6})=1$. From $T(a_{i+3})=T(b_i)=1$ and $T(b_{i+3})=2$, we get $T(b_{i+6})=2$. Now, from $T(a_{i+5})=T(b_{i+6})=2$ and $T(a_{i+6})=1$, we have $T(a_{i+7})=1$ and then $T(b_{i+7})=2$ which is a contradiction of $T(b_{i+4})=1$ and $T(a_{i+4})=T(b_{i+1})=T(b_{i+7})=2$.
\end{proof}
\subsection*{Perfect 2-colorings of $GP(n,3)$ with the matrix $A_4$:}
We show that just the graphs $GP(4m,3)$ among the graphs $GP(n,3)$ have a perfect 2-coloring with the matrix $A_4$. 
\begin{theorem}
\label{2}
All the graphs $GP(n,3)$, where $4 \mid n$, have a perfect 2-coloring with the matrix $A_4$. Also, there are no perfect 2-coloring of $GP(n,3)$, where $4\nmid n$, with this matrix.
\end{theorem}
\begin{proof}
For the first part, consider the mapping $T:V(GP(4m,3))\rightarrow \lbrace 1,2\rbrace$ by
\begin{center}
$T(a_{4i})=T(b_{4i+2})=1,\qquad\qquad\qquad\quad\;\;\,\qquad\qquad\qquad\qquad\qquad\quad\;\;\,$\\
$T(b_{4i})=T(a_{4i+1})=T(b_{4i+1})=T(a_{4i+2})=T(a_{4i+3})=T(b_{4i+3})=2.$
\end{center}
for $i \geq 0$. It can be easily checked that the given mapping is a perfect 2-coloring with the matrix 
$A_4$. \\
To prove the second part, contrary to our claim, suppose that there is a perfect 2-coloring of 
$GP(n,3)$ 
with the matrix 
$A_4$, say 
$T$.
with no restriction of generality, let 
$T(a_0)=1$. It follows that 
$T(a_1)=T(b_0)=T(a_{n-1})=T(b_{n-1})=2$.
From 
$T(a_1)=2$
and 
$T(a_0)=1$
we get 
$T(b_1)=T(a_2)=2$. Now, we should have two cases below.\\
\textbf{Case 1:} 
$T(b_2)=1$. It immediately gives 
\begin{center}
$T(a_{4i})=T(b_{4i+2})=1,\qquad\qquad\qquad\quad\;\;\,\qquad\qquad\qquad\qquad\qquad\quad\;\;\,$\\
$T(b_{4i})=T(a_{4i+1})=T(b_{4i+1})=T(a_{4i+2})=T(a_{4i+3})=T(b_{4i+3})=2.$
\end{center}
for $i \geq 0$. It clearly gives
$n=4m$ which is a contradiction of
$4\nmid n$.\\
\textbf{Case 2:}
$T(b_2)=2$. It immediately gives 
$T(a_3)=1$ and $T(b_3)=T(a_4)=2$. 
From 
$T(a_4)=2$
and 
$T(a_3)=1$,
we get
$T(b_4)=T(a_5)=2$. Then, from
$T(a_2)=T(b_2)=T(b_{n-1})=2$,
we have 
$T(b_5)=1$. So, we immidiately conclude that 
$T(a_6)=2$. Now, from
$T(b_0)=T(b_3)=2$ and
$T(a_3)=1$, we have 
$T(b_6)=2$. It gives 
$T(a_7)=1$ and then
$T(b_7)=2$ which is a contradiction of
$T(b_0)=T(b_4)=T(a_4)=2$.
\end{proof}
\subsection*{Perfect 2-colorings of $GP(n,3)$ with the matrix $A_5$:}
Here, we show that just the graphs $GP(5m,3)$, where $m\in \mathbb{N}$, among the graphs $GP(n,3)$ have a perfect 2-coloring with the matrix $A_5$.  
\begin{theorem}
The graphs $GP(5m,5t+2)$ and $GP(5m,5t+3)$, where $t\geq 0$,  have a perfect 2-coloring with the matrix
$A_5$. $GP(n,k)$ graphs for $n$ such that $5\nmid n$, have no perfect colorings with the matrix 
$A_5$.
\end{theorem}
\begin{proof}
For the first part, consider the mapping
$T:V(GP(5m,5t+2))\rightarrow \lbrace 1,2\rbrace$ by
\begin{center}
$T(a_{5i})=T(a_{5i+2})=T(a_{5i+3})=T(b_{5i})=T(b_{5i+1})=T(b_{5i+4})=2,$\\
$T(a_{5i+1})=T(a_{5i+4})=T(b_{5i+2})=T(b_{5i+3})=1,\qquad\qquad\qquad\qquad$\\
\end{center}
for $i\geq 0$. It can be easily checked that the given mapping gives a perfect 2-coloring with the matrix 
$A_5$. The mapping $T:V(GP(5m,5t+3))\rightarrow \lbrace 1,2\rbrace$ by the exactly above definition is also a perfect 2-coloring with the matrix 
$A_5$. Moreover, the second part can be proved by Proposition \ref{1}.
\end{proof}
Finally, we summerize the obtained results from enumerating the parameter matrices of $GP(n,k)$ in the following table.
\begin{table}[!ht]
\centering \caption{Eumerating the parameter matrices of $GP(n,k)$}
\begin{tabular}{|c|c|c|c|}
\hline
&$GP(n,2)$&$GP(n,3)$&$GP(n,k)$\\
\hline
$A_1$& all graphs&all graphs&all graphs\\
\hline
$A_2$& just $GP(3m,2)$& no graphs& ?\\
\hline
$A_3$& no graphs& just $GP(2m,3)$& ?\\
\hline
$A_4$& no graphs& just  $GP(4m,3)$& ?\\
\hline
$A_5$& just  $GP(5m,2)$& just $GP(5m,3)$&?\\
\hline
$A_6$& no graphs & just $GP(2m,3)$& just $GP(2m,2t+1)$\\
\hline
\end{tabular}
\end{table}

\appendix



\begin{thebibliography}{10}
\bibitem{8}
Alaeiyan Mehdi, Karami Hamed, Perfect 2-colorings of generalized Petersen graphs, Proc. Indian Acad. Sci. (Math. Sci.) Vol. 126, No. 3, August 2016, pp. 289–294.
\bibitem{9}
Alaeiyan Mehdi, Karami Hamed, Siasat Sajjad, Perfect 3-colorings of GP(5,2), GP(6,2), and GP(7,2), JOURNAL OF THE INDONESIAN MATHEMATICAL SOCIETY, Vol. 24, No. 2, October 2018, pp. 47-53.
\bibitem{10}
Alaeiyan Mohammad Hadi, Karami Hamed, Perfect 2-colorings of Platonic graphs, Proc. Iranian Journal of Nonlinear Analysis and Apllication, Vol. 8, No. 2, December 2017, pp. 29-35.
\bibitem{1}
S. V. Avgustinovich, I. Yu. Mogilnykh, "Perfect 2-colorings of Johnson graphs J(6, 3) and J(7, 3),"\textit{ 
Lecture Notes in Computer Science}, Volume \textbf{5228}  (2008)  11-19.
\bibitem{2}
S. V. Avgustinovich, I. Yu. Mogilnykh, "Perfect colorings of the Johnson graphs J(8, 3) and J(8, 4) with two colors" \textit{Journal of Applied and Industrial Mathematics}, Volume \textbf{5}  (2011) 19-30.
\bibitem{5}
D.G. Fon-Der-Flaass, "A bound on correlation immunity," \textit{Siberian Electronic Mathematical Reports Journal}, Volume \textbf{4} (2007) 133-135. 
\bibitem{4}
D. G. Fon-Der-Flaass, "Perfect 2-colorings of a hypercube," \textit{Siberian Mathematical Journal}, Volume \textbf{4} (2007) 923-930.
\bibitem{6}
D. G. Fon-der-Flaass, "Perfect 2-colorings of a 12-dimensional Cube that achieve a bound of correlation immunity," \textit{Siberian Mathematical Journal}, Volume \textbf{4} (2007) 292-295.
\bibitem{3}
A.L. Gavrilyuk and S.V. Goryainov, "On perfect 2-colorings of Johnson graphs J(v,3)," \textit{Journal of Combinatorial Designs}, Volume \textbf{21} (2013)  232-252.
\bibitem{7}
C. Godsil, "Compact graphs and equitable partitions," \textit{Linear Algebra and Its Application}, Volume \textbf{255} (1997) 259-266
\end{thebibliography}


\end{document}